\documentclass[a4paper,draft,12pt]{article}
\pdfoutput=1
\usepackage{amssymb,amsmath}
\usepackage{amsthm}
\usepackage{tikz}
 \newtheorem{thm}{Theorem}[section]
 
 \newtheorem{lem}[thm]{Lemma}
 \newtheorem{prop}[thm]{Proposition}
 \theoremstyle{definition}
 
 \theoremstyle{remark}
 \newtheorem{rem}[thm]{Remark}
 
 \numberwithin{equation}{section}

\usepackage{mathtools}

\DeclarePairedDelimiter\floor{\lfloor}{\rfloor}

\def\text{\mbox}

\begin{document}

\author{S.Capparelli and A. Del Fra}
\title{Chebyshev coordinates and Salem numbers}
\maketitle
\begin{abstract}
By expressing polynomials in the basis of Chebyshev polynomials, certain families of hyperbolic polynomials appear naturally. Some of these families  have all their roots in the interval $[-2,2]$.
In many cases the span of the family of polynomials thus found is greater than 4, and we show that they are the minimal polynomials of Salem numbers, possibly multiplied by some cyclotomic polynomials.  In addition, we show how to compute the limit of the largest and smallest roots.
\end{abstract}

\section{Introduction}
According to a classical result of Kronecker's, \cite{KRO}, a set of conjugate algebraic integers lying on the unit circle $|z|=1$ must be roots of unity. The unit circle $|z|=1$ is then transformed into the segment $-2\leq x\leq 2$ by the transformation $x=z+\frac{1}{z}$. Thus Kronecker obtains that any algebraic integer which lies with its conjugates in the interval $[-2,2]$ must be of the form $x=2\cos \frac{2k\pi}{m}$. We call the corresponding polynomials {\it polynomials of cosine type} or also  {\it Kronecker polynomials}.

In \cite{ROB1} (see also \cite{ROB3}) R. Robinson writes that P\'olya and I. Schur, \cite{SCH}, showed that a real interval of length less than 4 can contain only a finite number of sets of conjugate algebraic integers, and then proceeds to prove that any real interval of length greater than 4 contains an infinite number of sets of conjugate algebraic integers. He writes: ``The problem remains unsolved for intervals of length exactly 4, except when the end-points are rational integers, in which case there are infinitely many sets''.

In a 1964 paper, \cite{ROB2}, Robinson classified all irreducible polynomials with integer coefficients having only real roots, such that the difference between the largest and the smallest root, (the {\it span}), is less than 4, for degrees up to and including 8. He chooses the representative polynomials of each type in such a way that the average of the roots lies in $[0,\frac 12]$.

Recently, in \cite{CDS1}, Robinson's classification was extended up to degree 14, moreover a list up to degree 17 was obtained and conjectured to be complete.

In \cite{FRW} it was proved that the list is exhaustive up to degree 15. Finally, in \cite{EMRS}, using ideas from linear programming, the authors seem to suggest that the list of polynomials of degree 16 and 17 found in previous papers is indeed complete, though no proof of this is yet available. Moreover, the four authors were able to exhibit three polynomials of the desired type of degree 18 and, in spite of similar computations being conducted, no such polynomials of degree 19 and 20 were found, that were not of cosine type.
A related research is the one carried out by J. McKee in \cite{JMC} which uses integer symmetric matrices.
In many of the cited works Chebyshev polynomials seem to play an important role and in many cases several
heuristic methods have been used to help pinpoint the elusive hyperbolic polynomials. Chebyshev polynomials were used also in recent work by the first author in \cite{CM1} and \cite{CM2}.
So we decided to choose Chebyshev polynomials as a basis to express the polynomials already classified.
 It turns out that when expressed in terms of Chebyshev polynomials often the polynomials in the list appear
 to have coefficients with small absolute value and with a certain regularity. Indeed, with hindsight, using these coordinates it is possible to find by brute force a good number of the Robinson's polynomials with little effort. Some polynomials in Robinson's list are very simple linear combinations of Chebyshev polynomials. By choosing only coefficients $1,-1,0$ one can describe 39 polynomials out of the 96 in Robinson's original list. For example, the polynomial $x^6-x^5-6x^4+5x^3+9x^2-6x-1$ has Chebyshev coordinates  $[1,-1,0,0,0,-1,1]$  and
 $x^8-8x^6-x^5+20x^4+4x^3-16x^2-3x+2$ has coordinates  $[0,-1,0,-1,0,-1,0,0,1]$. We also noticed that those polynomials which have all roots in the interval $[-2,2]$, i.e., those classified by Kronecker, often have these simple kind of coordinates.
 For example, $x^8-x^7-7x^6+6x^5+15x^4-10x^3-10x^2+4x+1$ has coordinates  $[1,-1,1,-1,1,-1,1,-1,1]$ and $x^8-8x^6+20x^4-16x^2+2$ has coordinates  $[0,0,0,0,0,0,0,0,1]$. Other regularities were observed as, for example, we almost always have alternating signs and a weak monotonicity, in absolute value. For example, $x^8-3x^7-5x^6+18x^5+7x^4-33x^3-3x^2+18x+1$ has coordinates $[7,-6,6,-6,5,-3,3,-3,1]$.
  A glance at the following table of Chebyshev coordinates of the degree 8 Robinson's polynomials ordered by increasing span will show the evident simplification and regularity afforded by the Chebyshev basis (the asterisk indicates a polynomial of cosine type):

  {\small
\begin{center}
    \begin{tabular}{ | l| l | l | l | l |}
    \hline
 No.&    Chebyshev Coordinates  \\
    \hline

    8a&  $[31,-30,27,-22,17,-11,7,-4,1]$ \\
    \hline
    8b*&  $[1,-1,1,-1,1,-1,1,-1,1]$ \\
    \hline
    8c&  $[29,-27,25,-20,16,-11,7,-4,1]$ \\
    \hline
    8d&  $[17,-16,15,-12,10,-7,4,-3,1]$ \\
    \hline
    8e* & $[0,0,0,0,0,0,0,0,1]$ \\
    \hline
    8f&  $[0,-1,0,-1,0,-1,0,0,1]$ \\
    \hline
    8g&  $[13,-12,12,-10,8,-6,4,-3,1]$ \\
    \hline
    8h& $[13,-13,12,-10,8,-6,4,-3,1]$ \\
    \hline
    8i&  $[5,-5,4,-4,3,-3,2,-2,1]$ \\
    \hline
    8j&  $[7,-6,6,-5,4,-4,2,-2,1]$ \\
    \hline
    8k* &  $[1,0,0,0,-1,0,0,0,1]$ \\
    \hline
    8l&  $[17,-16,14,-12,10,-7,6,-4,1]$ \\
    \hline
    8m&  $[7,-6,6,-6,5,-3,3,-3,1]$ \\
    \hline
    8n&  $[3,-3,3,-3,3,-3,2,-2,1]$ \\
    \hline
    8o&  $[1,-4,1,-3,1,-2,1,-1,1]$ \\
    \hline
    8p*&  $[-1,0,0,0,0,0,0,0,1]$ \\
    \hline
    8q&  $[11,-10,10,-9,8,-6,4,-3,1]$ \\
    \hline
    8r*&  $[-1,0,-1,0,0,0,1,0,1]$ \\
    \hline
    8s&  $[1,-2,1,-2,1,-2,1,-1,1]$ \\
    \hline
    8t&  $[5,-3,3,-3,3,-2,1,-2,1]$ \\
    \hline
    8u&  $[3,-5,5,-2,4,-3,1,-2,1]$ \\
    \hline
    8v&  $[3,-4,4,-3,3,-2,1,-2,1]$ \\
    \hline
    8w&  $[3,-1,2,-2,1,-2,0,-1,1]$ \\
    \hline
    8x&  $[21,-20,18,-15,13,-10,7,-4,1]$ \\
    \hline
    8y&  $[3,-4,2,-4,1,-3,1,-1,1]$ \\
    \hline
    8z& $[5,-5,5,-3,4,-3,1,-2,1]$ \\
    \hline

    \end{tabular}
\end{center}

}

  With this method, one of the three polynomials of degree 18 found in \cite{EMRS}, could have been found as it  is exactly of this form:
 \small{
 \begin{equation*}
[15, -15, 15, -14, 14, -13, 12, -11, 10, -9, 8, -7, 6, -5, 4, -3, 2, -2, 1]
\end{equation*}
}

So our hope was that this new basis could shed new light on this problem. We thus decided to study some families of polynomials that have some regularity in their Chebyshev coordinates.

We  determine  certain conditions on the Chebyshev coefficient that  guarantee that the given polynomial has roots in the interval  $[-2,2]$, or, as we often say, be a polynomial of Kronecker type.

A certain generalization of roots of unity is given by Salem numbers (\cite{Sal1},\cite{Sal2}, see also \cite{SMY}). A Salem number is defined as an algebraic integer $\tau>1$ of degree at least 4, conjugate to $\tau^{-1}$, all of whose conjugate different from $\tau$ and $\tau^{-1}$ lie on the unit circle.
We find certain families of polynomials that are essentially minimal polynomials of Salem numbers. To show this we use the classic construction of Salem starting from Pisot numbers.

\section{Preliminaries}\label{prelim}

Consider  Chebyshev polynomials as defined in \cite{ROB1} for $n>0$
\begin{equation}
  T_n(x)=x^n+\sum_{k=1}^{\floor{\frac{n}{2}}}(-1)^k\frac nk \binom{n-k-1}{k-1}x^{n-2k}
\end{equation}
and $T_0(x)=1$. Notice that these are the Chebyshev polynomials in $[-2,2]$, namely $T_n(2\cos \theta)=2\cos (n\theta)$.

Let $\mathcal{B}=(T_0(x),T_1(x),T_2(x),\ldots)$ be the  ordered basis   made up of these monic polynomials and let  $m$ be the matrix of the change of basis
from the standard basis $\mathcal N=
(1,x,x^2,x^3,\ldots )$ to $\mathcal B$.

The matrix   $m$ is infinite, upper triangular with 1 on the main diagonal.
Its entries are
 $$ m_{0,0}=1, \quad m_{0,2j+1}=0, \quad j=0,1,\ldots,\quad m_{0,2j}=(-1)^{j}2,\quad j=1,2,\ldots$$
 $$m_{2i,2j}= (-1)^{i+j}\binom{i+j-1}{2i-1}\frac{j}{i},\quad m_{2i,2j+1}=0,\quad i=1,2,\ldots, \quad j=0,1,\ldots,$$
 $$m_{2i+1,2j}= 0, m_{2i+1,2j+1}= (-1)^{i+j}\binom{i+j}{2i}\frac{2j+1}{2i+1}, i=0,1,\ldots,\quad j=0,1,\ldots,$$

Its inverse matrix  $b=m^{-1}$ is also infinite and upper triangular and its entries are
$$
b_{2i,j}=\frac{1+(-1)^j}{2}\binom{j}{\frac{j+2i}{2}}, i\geq 0, j\geq 0,
$$

$$
b_{2i+1,j}=\frac{1-(-1)^j}{2}\binom{j}{\frac{j-2i-1}{2}}, i\geq 0, j\geq 0.
$$

In what follows we are going to  examine some families of polynomials whose coordinates in the Chebyshev basis are particularly simple and regular.
In general, for a family of polynomials, we would like to ascertain whether they have  all their roots in the critical interval $[-2,2]$ and if not, whether the span is ``small'' or not.

As mentioned in the Introduction, we are going to use the transformation  $x=z+z^{-1}$, where $x\in \mathbb{R}$ and $z\in \mathbb{C}$. Notice that
  $x\in \mathbb{R}$ implies that either $z\in \mathbb{R}$, and so $|x|\ge 2$, or $z$ lies on the unit circle and  $|x| \le 2$.

Recall that for the  Chebyshev polynomials of the first kind, this transformation gives, for $n>0$,
$$
T_n(x)=T_n(z+z^{-1})=z^n+z^{-n}.
$$

{\bf Notation}: Given a function $f(x)$, we shall write $\widetilde{f}(z)=f(z+z^{-1})$. Obviously, $\widetilde{T}_0(z)=1$.
\section{Some Kronecker families}
For $k,s$ nonnegative integers, and an integer $n>1$,
 let $P_{k,n}^{(s)}(x)$ be a polynomial with Chebyshev coordinates
$$[\underbrace{0,\ldots, 0}_{s},1,\underbrace{0,\ldots, 0}_{k},1,\underbrace{0,\ldots, 0}_{k},1,\ldots,1,\underbrace{0,\ldots, 0}_{k},1]$$
 where $n$ is the number of 1, and $n-1$ the number of blocks of $k$ zeros.

We have
$$\widetilde{P}_{k,n}^{(s)}(z)=\widetilde{T}_s(z)+\widetilde{T}_{k+1+s}(z)+\cdots+\widetilde{T}_{(n-1)(k+1)+s}(z)$$
$$= \sum_{j=0}^{n-1}(z^{j(k+1)+s}+z^{-(j(k+1)+s)}).$$

Hence
\begin{equation*}
  \begin{split}
    &z^{(n-1)(k+1)+s}\widetilde{P}_{k,n}^{(s)}(z)\\
    & = z^{(n-1)(k+1)+s}(z^s+z^{-s}+\cdots+z^{(n-1)(k+1)+s}+z^{-((n-1)(k+1)+s)})\\
    &=1+z^{k+1}+\cdots+z^{(n-1)(k+1)}+z^{(n-1)(k+1)+2s}+\cdots+z^{2(n-1)(k+1)+2s}\\
    &=(1+z^{(n-1)(k+1)+2s})(1+z^{k+1}+\cdots+z^{(n-1)(k+1)})\\
    &= (1+z^{(n-1)(k+1)+2s})\frac{z^{n(k+1)}-1}{z^{k+1}-1}.
  \end{split}
\end{equation*}

Then the roots of $z^{(n-1)(k+1)+s}\widetilde{P}_{k,n}^{(s)}(z)$ are  all on the unit circle. So, translating back to $x$,  $P_{k,n}^{(s)}(x)$ is hyperbolic, its roots are in $[-2,2]$ and is therefore of Kronecker type.

Fixing $k$ and $s$, as $n$ approaches $+\infty$,  the roots become dense in the unit circle and so the largest root tends to 2 and the smallest to $-2$.

\begin{rem}
  The special case $s=0$, $k=1$, any $n$, 
  $P_{1,n}^{(0)}(x)$ has coordinates
$$[1,0,1,0,1,\ldots,0,1,0,1]$$
 where $n$ is the number of 1, and $n-1$ the number of  zeros. This is just a sum of Chebyshev polynomials of the first kind.
  The case $s=0$ or $s=1$, $k=1$, any $n$, gives Chebyshev polynomials of the second kind.
  \end{rem}

Let $A_{n}(x)$ be a degree $n$ polynomial with  Chebyshev coordinates $$[2,2,2,\ldots,2,2,1 ].$$
We have
$$\widetilde{A}_{n}(z)=2\widetilde{T}_0(z)+2\widetilde{T}_1(z)+2\widetilde{T}_2(z)+\cdots+2\widetilde{T}_{n-1}(z)+\widetilde{T}_{n}(z)$$
$$=2(1+\sum_{j=1}^{n-1}(z^{j}+z^{-j}))+z^{n}+z^{-n}.$$
Hence

\begin{equation*}
  \begin{split}
    z^{n}\widetilde{A}_{n}(z) &=
    2(1+z+z^2+\cdots+z^{2n-1})+z^{2n}-1\\
    &=2\frac{z^{2n}-1}{z-1}+z^{2n}-1=(z^{2n}-1)\frac{z+1}{z-1},
  \end{split}
\end{equation*}
which again has all roots on the unit circle.
Hence $A_{n}(x)$ is a Kronecker polynomial.

Let $B_{2n}(x)$ be a  polynomial of even degree $2n$ with Chebyshev coordinates
$$
[2,1,2,1,\ldots,2,1,1 ].
$$

Then
\begin{equation*}
  \begin{split}
    \widetilde{B}_{2n}(z)
    &=2\widetilde{T}_0(z)+\widetilde{T}_1(z)+2\widetilde{T}_2(z)+\widetilde{T}_3(z)+\cdots
    +2\widetilde{T}_{2n-2}(z)+\widetilde{T}_{2n-1}(z)+\widetilde{T}_{2n}(z)\\
    &=2(1+\sum_{j=1}^{n-1}(z^{2j}+z^{-2j}))+\sum_{j=1}^{n}(z^{2j-1}+z^{-(2j-1)})+z^{2n}+z^{-2n}.
  \end{split}
\end{equation*}
and so

$$z^{2n}\widetilde{B}_{2n}(z) =$$
$$2(1+z^2+\cdots+z^{4n-2})+z(1+z^2+\cdots+z^{4n-2})+z^{4n}-1=$$
$$(2+z)\frac{z^{4n}-1}{z^2-1}+z^{4n}-1= (z^{4n}-1)(\frac{z^2+z+1}{z^2-1}).$$
which again has all roots on the unit circle.
Hence $B_{2n}(x)$ is a Kronecker polynomial.

\medskip
The case  $B_{2n+1}(x)$, a polynomial of degree $2n+1$ and Chebyshev coordinates $$[1,2,1,2,1,\ldots,2,1,1 ]$$ is completely analogous.

Both in the case $\widetilde{A}_n(z)$ and $\widetilde{B}_n(z)$
we can see that, as $n$ approaches $+\infty$,  the roots become dense in the unit circle and so the largest root of $A_n(x)$ and $B_n(x)$ tends to 2 and the smallest to $-2$.

\section{One more family}

Let $P_{2n+k-1}(x)$ be the family of polynomials of degree $2n+k-1$ and coordinates
$$
[-1,0,-1,0,\ldots,-1,\underbrace{0,\ldots,0}_{k},1],
$$
 where $n$ is the number of $-1$ and $k$ a nonnegative integer.

$$\widetilde{P}_{2n+k-1}(z)=-\widetilde{T}_0(z)-\widetilde{T}_2(z)-\widetilde{T}_4(z)+\cdots-\widetilde{T}_{2n-2}(z)+\widetilde{T}_{2n+k-1}(z)$$
$$=-(1+z^{2}+z^{-2}+z^{4}+z^{-4}+\cdots+z^{2n-2}+z^{-(2n-2)})+z^{2n+k-1}+z^{-(2n+k-1)},$$
and so
$$z^{2n+k-1}(z^2-1)\widetilde{P}_{2n+k-1}(z) =$$
$$[-(z^{k+1}+z^{k+3}+\cdots+z^{4n+k-3})+1+z^{4n+2k-2}](z^2-1)=$$
$$ z^{4n+2k}-z^{4n+2k-2}-z^{4n+k-1}+z^{k+1}+z^2-1=$$
$$z^{4n+k-1}(z^{k+1}-z^{k-1}-1)+z^{k+1}+z^{2}-1.$$

When $k=0$ this last step can be rewritten as $z^{4n-2}(z^{2}-z-1)+z^{2}+z-1$.

We now observe that
\begin{prop}\label{same module}
For $k\geq 1$ the polynomials $ z^{k+1}-z^{k-1}-1$ and $z^{k+1}+z^{2}-1$ have equal absolute value on the unit circle. The same happens for $z^2-z-1$ and $z^2+z-1$ (case $k=0$).
\end{prop}

\begin{proof}

Let $k\geq 1$ and set $z^{k+1}-z^{k-1}-1=f(z)$ and $z^{k+1}+z^{2}-1=g(z)$, we notice that $g(z)=-z^{k+1}f(z^{-1})$. For $|z|=1$ we have $f(z^{-1})=f(\bar{z})$. Hence for $|z|=1$ we have
$$|g(z)|=|z^{k+1}||f(z^{-1})|=1|f(\bar{z})|=|\overline{f(z)}|=|f(z)|.$$
A similar argument holds for the two polynomials $z^2-z-1$ and $z^2+z-1$.
\end{proof}

In general, given two polynomials $f(x)$, $g(x)$, where $g(z)=\pm z^{degf}f(z^{-1})$, $f(x)$ and $g(x)$  have the same absolute value on the unit circle.

We then have
\begin{thm}
 For $k=0,1,2,3,5,7$, the polynomials $P_{2n+k-1}(x)$ are hyperbolic with all roots in the interval $[-2,2]$ except for one $x_M>2$, moreover $\lim_{n\to \infty} x_M=z_0+z_0^{-1}$, where $z_0$ is the largest real root of the polynomial $z^{k+1}-z^{k-1}-1$ ($z^2-z-1$ for $k=0$), while the smallest root $x_m$ approaches $-2$.
\end{thm}
\begin{proof}
  Let
  \begin{equation}
    \begin{split}
      R(z)&=z^{2n-1}(z^2-1)\widetilde{P}_{2n+k-1}(z)=\\
     &=\begin{cases}
    z^{4n-2}(z^{2}-z-1)+z^{2}+z-1, &k=0\\
    z^{4n+k-1}(z^{k+1}-z^{k-1}-1)+z^{k+1}+z^{2}-1,&k\geq1
  \end{cases}.
    \end{split}
  \end{equation}

   For $k=0$, $z^{2}-z-1$  is the minimal polynomial of a Pisot number. Since Proposition \ref{same module} guarantees that $|z^{2}-z-1|=|z^{2}+z-1|$ on the unit circle, then a standard result of Salem, \cite{Sal3}, (see also \cite{Sal1} and \cite{Sal2}).
   implies that, for large enough $n$, $R(z)$ is the minimal polynomial of a Salem number, possibly multiplied by some cyclotomic polynomials.

It is then clear that the largest root of $R(z)$ approaches, as $n$ goes to infinity, the largest root $z_0(>1)$ of $z^2-z-1$, another root approaches $\frac{1}{z_0}$, while all the other roots are on the unit circle, and, as $n$ goes to $\infty$,  they become dense there. Thus, translating back to $x$, all roots $x$ are real, the least root tends to $-2$ and the greatest to $z_0+z_0^{-1}$.

   For $k=2$, $z^{k+1}-z^{k-1}-1=z^3-z-1$  is the minimal polynomial of a Pisot number and we can repeat the same argument as above.

   \medskip
   For odd $k$, setting $k=2h+1$, the polynomial $z^{k+1}-z^{k-1}-1=z^{2h+2}-z^{2h}-1$ cannot be a minimal polynomial of a Pisot number, so we set $z^2=y$ and obtain $y^{h+1}-y^h-1$.
$R(z)$, as a function of $y$, is $y^{2n+h}(y^{h+1}-y^h-1)+y^{h+1}+y-1$. For $h=0,1,2,3$ these are minimal polynomials of a Pisot number so an argument similar to the even cases $k=0, 2$ shows that
$y^{2n+h}(y^{h+1}-y^h-1)+y^{h+1}+y-1$ is a minimal polynomial of a Salem number. Hence $R(z)$ has pairs of opposite roots on the unit circle except for two on the real axis, one less than -1 and the other greater than 1. As $n$ goes to infinity the positive root tends to the largest root of  $z^{2h+2}-z^{2h}-1$. Thus, translating back to $x$, all roots $x$ are real,  the greatest tends to $z_0+z_0^{-1}$ (and the smallest to $-(z_0+z_0^{-1})$ .
\end{proof}

\begin{rem}
  For $k$ even and greater than  2 and for $k$ odd and greater than 7 unfortunately  $z^{k+1}-z^{k-1}-1$ is not the minimal polynomial of a  Pisot number. For even $k$, besides having a real root greater than 1, it also has two complex conjugates roots outside of the unit circle. For $k$ odd, after setting $z^2=y$, we have an analogous  situation. Thus Salem construction in theory is not possible. However,  although the polynomials of the family are not hyperbolic, for even  $k>2$, they still have a largest real root that approaches the largest real root of $z^{k+1}-z^{k-1}-1$  and the smallest root approaches $-2$ while, for odd $k$, the largest and smallest root, that are opposite, approach the largest and smallest, respectively, of $z^{k+1}-z^{k-1}-1$.

Actually, these polynomials are almost hyperbolic, in the sense that, for even  $k$, they have only a pair of complex conjugate roots, and for odd $k$ only two opposite pairs of complex conjugate roots.
\end{rem}

\section{Two-parameter family}

In this section, we consider the family of polynomials $P^{(n)}_{(h_1,h_2)}(x)$ depending on two integer parameters $h_1,h_2$ with $1<h_1\leq h_2$, with Chebyshev coordinates
$$[1,\underbrace{-h_1,h_2,-h_1,h_2,\ldots, -h_1,h_2}_{2n},1]$$
 of degree $2n+1$ in $x$. We want to rewrite this polynomial in a suitable way.  For example, for $n=1$,

$$
\widetilde{P}^{(1)}_{(h_1,h_2)}(z)=\widetilde{T}_0(z)-h_1\widetilde{T}_1(z)+h_2\widetilde{T}_2(z)+\widetilde{T}_3(z)
$$
$$
=1-h_1(z+z^{-1})+h_2(z^2+z^{-2})+(z^3+z^{-3})
$$
Add and subtract $h_2$:
$$
\widetilde{P}^{(1)}_{(h_1,h_2)}(z)=z^3+z^{-3}+(1-h_2)-h_1(z+z^{-1})+h_2(1+z^2+z^{-2})
$$
multiply by $z^3(z^2-1)$:
$$
z^3(z^2-1)\widetilde{P}^{(1)}_{(h_1,h_2)}(z)=
$$
$$
(z^2-1)\left[z^6+1-(h_2-1)z^3-h_1(z^4+z^2)+h_2(z^3+z^5+z)\right]=
$$

$$
z^6\left[z^2+h_2 z-(h_1+1)\right]-(h_2-1)z^3(z^2-1)+(1+h_1)z^2-h_2z-1=
$$
$$
z^3\left\{z^3\left[z^2+h_2z-(h_1+1)\right]- \frac{(h_2-1)}{2}(z^2-1)\right\}
$$
$$
- \frac{(h_2-1)}{2}z^3(z^2-1)+(1+h_1)z^2-h_2z-1
$$
splitting the central summand  $-(h_2-1)z^{2n+1}(z^2-1)$ in two halves and adding one half to the first summand and the other half to the second summand.

This is the computation for $n=1$.

For $n\geq 1$ we can analogously obtain
\begin{equation}\label{rewrite1}
  \begin{split}
    &z^{2n+1}(z^2-1)\widetilde{P}^{(n)}_{(h_1,h_2)}(z)\\
    &=z^{2n+1}\left\{z^{2n+1}\left[z^2+h_2z-(h_1+1)\right]- \frac{(h_2-1)}{2}(z^2-1)\right\}\\
    &- \frac{(h_2-1)}{2}z^{2n+1}(z^2-1)+(1+h_1)z^2-h_2z-1\\
      \end{split}
\end{equation}

which we rewrite as
\begin{equation}\label{rewrite2}
  z^{2n+1}(z^2-1)\widetilde{P}^{(n)}_{(h_1,h_2)}(z)=z^{2n+1}Q(z)-z^{2n+3}Q(z^{-1})
\end{equation}
 where $Q(z)=z^{2n+1}\left[z^2+h_2z-(h_1+1)\right]-\frac{(h_2-1)}{2}(z^2-1)$. If  $Q(z)$ were the minimal polynomial of a  Pisot number then, by Salem construction,  $z^{2n+1}\widetilde{P}^{(n)}_{(h_1,h_2)}(z)$ is, (for sufficiently large $n$) the minimal polynomial of a  Salem number, possibly multiplied by some cyclotomic polynomials.

 If  $h_2$ is even, however, $Q(z)$ does not have integer coefficients and so it is not the minimal polynomial of a Pisot number. However  $z^{2n+1}\widetilde{P}^{(n)}_{(h_1,h_2)}(z)$ has integer  coefficients, as long as $Q(z)$ has a root greater than 1  in absolute value and all the other are inside the unit disc, then the Salem construction still works.

Set $U(z)=z^2+h_2z-(h_1+1)$. Since $U(-1)=-h_1-h_2<0$ (recall that both parameters are assumed  positive), while $U(1)=h_2-h_1\geq 0$, (by assumption $h_1\leq h_2$), we know that $U(z)$ has a real root $\alpha<-1$  and the other  in $(-1,1]$. So $\alpha$ is the opposite of a Pisot number.   In the case $h_1=h_2$ we have $U(z)=(z-1)(z+h_1+1)$.

\begin{lem}
  $$
\left|\frac{h_2-1}{2}(z^2-1)\right|\leq |U(z)|
$$
\end{lem}
 \begin{proof}
   For $z=a+ib$ on the unit circle we have
   $$
  |U(z)|^2- \left|\frac{h_2-1}{2}(z^2-1)\right|^2=(h_1-h_2a)^2+(4h_1+2h_2+3)b^2\geq 0.
   $$
 \end{proof}

 \begin{thm}
For every pair of integers $(h_1,h_2)$, with $0<h_1\le h_2$, the polynomials of the family with   Chebyshev coordinates $$[1,-h_1,h_2,-h_1,h_2,\ldots,-h_1,h_2,1]$$  are hyperbolic with all roots in the interval $[-2,2]$ except for one $x_m<-2$, moreover $\lim_{n\to \infty} x_m=z_0+z_0^{-1}$, where $z_0$ is the smallest real root of the polynomial $z^{2}+h_2z-(h_1+1)$, while the largest root $x_M$ approaches $2$.
\end{thm}

\begin{proof}
A standard argument, based on  Rouch\'e's theorem and because of the lemma,  shows that  $Q(z)$ in \eqref{rewrite2} has a real root less than  $-1$ and all the others in  $|z|<1$, in other words it  corresponds to the opposite of a Pisot number  (in an extended sense if $h_2$ is even as the coefficients may not be integers). Therefore $z^{2n+1}(z^2-1)\widetilde{P}^{(n)}_{(h_1,h_2)}(z)$ is the minimal polynomial of the opposite of a Salem number, possibly multiplied by some cyclotomic polynomials. From \eqref{rewrite1} one can deduce that as $n$ goes to $+\infty$ the Salem number approaches the smallest root $z_0$ of $z^2+h_2z-(h+1)$ and the other roots become dense on the unit circle. Translating back to $x$ we see that $P^{(n)}_{(h_1,h_2)}(x)$ is hyperbolic, with one root less than $-2$ and all the other in $[-2,2]$, and as $n$ goes to $+\infty$ the smallest root tends to $x_m=z_0+\frac{1}{z_0}$ and the largest $x_M$ tends to $2$.
\end{proof}

\begin{rem}
  It can be shown that the limit of the span, namely $2-z_0-z_0^{-1}$, is the largest root of
  the resultant with respect to $z$ of $z^2+h_2z-h_1-1$ and $(x-(2-z-z^{-1}))z$:
  $$
  -(h_1+1)x^2+(4(h_1+1)+h_1h_2)x+(h_2-h_1)^2
  $$
\end{rem}

\begin{rem}
  Elementary calculations also show that $\displaystyle\lim_{n\to \infty} x_m$ is the negative root of the polynomial
$$
(h_1+1)x^2+h_1h_2 x-[h_2^2+(h_1+2)^2].
$$
\end{rem}

\section{Case of  three parameters}

Consider the family of polynomials $P^{(n)}_{(h_1,h_2,h_3)}(x)$ of degree $6n+1$ depending on the integer parameters $h_1,h_2, h_3,$ with $1<h_1\leq h_2\leq h_3$, and Chebyshev coordinates
$$[1,-h_1,h_2,-h_3,h_1,-h_2,h_3,\ldots,-h_1,h_2,-h_3,h_1,-h_2,h_3,1].$$

We have

\begin{equation}
  \begin{split}
    z^{6n+1} (z^3+1)\widetilde{P}^{(n)}_{(h_1,h_2,h_3)}(z)=z^{6n+1}[z^{6n+1}(z^3+h_3z^{2}-h_2z+h_1+1)\\
    +1-h_3+(h_2-h_1)z
    +(h_2-h_1)z^{2}+(1-h_3)z^3]\\
    +(h_1+1)z^3-h_2z^2+h_3z+1
  \end{split}
\end{equation}
Setting
\begin{equation}
  \begin{split}
    Q(z)&=z^{6n+1}(z^3+h_3z^{2}-h_2z+h_1+1)\\
    &+\frac{1-h_3+(h_2-h_1)z
    +(h_2-h_1)z^{2}+(1-h_3)z^3}{2}
  \end{split}
\end{equation}

one gets
    \begin{equation}\label{Salem}
      z^{6n+1} (z^3+1)\widetilde{P}^{(n)}_{(h_1,h_2,h_3)}(z)=z^{6n+1}Q(z)+z^{{6n+3}}Q(z^{-1})
        \end{equation}

\begin{lem}\label{minimal} The polynomial  $$U(z)=z^3+h_3z^{2}-h_2z+h_1+1$$  has a real root $\alpha<-1$ and the others in the unit disc. (So $\alpha$ is a Pisot number).
\end{lem}
\begin{proof}
Notice that the polynomial $U(z)$ computed at $-(h_1+1)$ is certainly positive:
  $$
  U(-(h_1+1))=(h_1+1)(-h_1^2-2h_1+h_3h_1+h_3+h_2)>0
  $$
 then a real root $\alpha$ must be smaller than $-(h_1+1)$.

 Let $\beta, \gamma$ the other two roots. Then $|\alpha \beta \gamma|=h_1+1$ so that
 \begin{equation}\label{inequality}
   |\beta \gamma|=\frac{h_1+1}{|\alpha|}<1.
 \end{equation}

Suppose that $\gamma=\bar{\beta}$,  hence $|\beta|^2=\frac{h_1+1}{|\alpha|}<1$.
Next, suppose $\beta$ and $\gamma$ are real and   $|\beta|< |\gamma|$. From \eqref{inequality}, it follows that $-1<\beta<1$. Notice also that $U(-1)=h_1+h_2+h_3>0$. So $\beta<\gamma$.

Now,  $U(1)=2+h_1-h_2+h_3>0$, therefore  $\beta<\gamma<1$. In all cases $\beta, \gamma$ are in the unit disc and so we have the desired conclusion.
\end{proof}
\begin{lem}\label{unit}
  Set  $$U(z)=z^3+h_3z^{2}-h_2z+h_1+1$$ and $$V(z)=\frac{1-h_3+(h_2-h_1)z
    +(h_2-h_1)z^{2}+(1-h_3)z^3}{2},$$ then on the unit circle we have
        $ |  U (z) |^2 \geq | V (z) |^2$ .
\end{lem}
\begin{proof}
     Setting $z = x + iy$,
on the unit circle we have
\begin{equation*}
  \begin{split}
    | U (z) |^2 =
 a (x) = (x^3 - 3 x (1 - x^2) + h_3 (2 x^2 - 1)\\
  -
       h_2 x + h_1 + 1)^2 + (1 - x^2) (4 x^2  + 2 h_3 x -
        h_2- 1 )^2
  \end{split}
\end{equation*}
         and
 \begin{equation*}
   \begin{split}
     | V (z) |^2 =   b (x) = \frac 14(((h_3 - 1) (4 x^3 - 3 x + 1) + (h_1 -h_2) (2 x^2 + x - 1))^2\\ + (1 -
           x^2) ((h_3 - 1) (4 x^2 - 1) + (h_1 - h_2) (2 x + 1))^2)
   \end{split}
 \end{equation*}
  Moreover, setting  $f(x)=a(x)-b(x)$ we must show that  $f(x)\geq 0$ for $-1\leq x\leq 1$.
    It turns out that  $f(x)$ is a third degree polynomial function,  precisely
    \begin{equation*}
      \begin{split}
       f(x)=(-2 h_3^2+4 h_3+8 h_1+6)x^3+2 \left(h_2 \left(h_3-3\right)+2
   h_3+h_1 \left(h_3+1\right)\right)x^2\\
   +\frac{1}{2}
   \left(-h_1^2-2 \left(h_2+h_3+5\right) h_1-h_2^2+3 h_3^2-2 h_3-2 h_2 \left(h_3+3\right)-9\right)x\\
   +\frac{1}{2} \left(h_1^2+2 \left(h_2-h_3+1\right) h_1+h_2^2+h_3^2-2 h_2 \left(h_3-3\right)-2 h_3+3\right).
      \end{split}
    \end{equation*}
    Reordering the expression according to the powers of the parameters $h_i$, we get
    \begin{equation*}
      \begin{split}
       f(x)= \frac 12(12x^3-9x+3)+(8x^3+2x^2-5x+1)h_1+(-6x^2-3x+3)h_2\\
        +(4x^3+4x^2-x-1)h_3
        +(-x+1)h_1h_2+(2x^2-x-1)h_1h_3\\
        +(2x^2-x-1)h_2h_3
        +\frac 12 (-x+1)h_1^2+\frac 12 (-x+1)h_2^2
        +\frac 12(-4x^3+3x+1)h_3^2.
      \end{split}
    \end{equation*}
Now consider $2f(x)$. One  has:
$$
2f(x)=(1-x)(h_1+h_2-(1+2x)h_3)^2+A_0(x)+A_1(x)h_1+A_2(x)h_2+A_3(x)h_3
$$
where we set  $A_0(x)=3-9x+12x^3$, $A_1(x)=2(1-5x+2x^2+8x^3)$, $A_2(x)=2(3-3x-6x^2)$, $A_3(x)=-2(1-x-2x^2)(1+2x)$.
Since in the interval $[-1,-\frac 12]$ one has $A_i(x)\geq 0, i=0,1,2,3$, in this same interval $2f\geq 0$.
One cannot say the same thing in the interval $[-\frac 12, 1]$. The structure of the functions  $A_i(x)$ suggests to introduce in the interval $[-\frac 12, 1]$ an auxiliary function $\alpha(x)$ and to decompose $2f$ in the following fashion as sum of  positive quantities.
\begin{equation}\label{somma}
  \begin{split}
    (1-x)(h_1+h_2-(1+2x)h_3+\alpha(x))^2\\
    +R_0(x)+R_1(x)h_1+R_2(x)h_2+R_3(x)h_3
  \end{split}
\end{equation}
where we set  $R_0(x)=A_0(x)-(1-x)\alpha(x)^2$, $R_1(x)=A_1(x)-2(1-x)\alpha(x)$, $R_2(x)=A_2(x)-2(1-x)\alpha(x)$, $R_3(x)=A_3(x)+2(1-x)(1+2x)\alpha(x)$.

Our task is to decompose the interval $[-\frac 12 , 1]$ in subintervals where for each subinterval we choose a suitable, possibly different, $\alpha(x)$ such that the sum
$$
R(x)=R_0(x)+R_1(x)h_1+R_2(x)h_2+R_3(x)h_3
$$
 is greater than or equal to zero.

In the interval $[-\frac 12 ,0]$, by choosing $\alpha(x)=1-2x$, we have  $R_i(x)\geq 0, i=0,1,2,3$ and therefore $R(x)\geq 0$.

In the interval $[0,\frac 17]$, we choose $\alpha(x)=2-4x$, and we have $R_0(x)\leq0$,  $R_1(x)<0$, $R_2(x)>0$, $R_3(x)>0$ so that
\begin{equation*}
  \begin{split}
    R(x)=R_0(x)+R_1(x)h_1+R_2(x)h_2+R_3(x)h_3\\
     \geq(R_0(x)+R_1(x)+R_2(x)+R_3(x))h_1
  \end{split}
\end{equation*}
It follows that $R(x)\geq 0$
since in $[0,1/7]$ the function $R_0(x)+R_1(x)+R_2(x)+R_3(x)$ is nonnegative.

In the interval $[\frac 17,\frac 14]$, we choose again  $\alpha(x)=2-4x$. Now we have  $R_0(x)>0$, $R_1(x)<0$, $R_2(x)>0, R_3(x)>0$, and $R_1(x)+R_2(x)+R_3(x)\geq 0$. Hence
\begin{equation*}
  \begin{split}
    R(x)=R_0(x)+R_1(x)h_1+R_2(x)h_2+R_3(x)h_3\geq\\
      R_0(x)+(R_1(x)+R_2(x)+R_3(x))h_1\geq 0.
  \end{split}
\end{equation*}

Next, in the interval $[\frac 14,\frac 12]$, we choose $\alpha(x)=3-6x$, and we have  $R_0(x)\leq 0,  R_1(x)\leq 0, R_2(x)\geq0, R_3(x)>0$
\begin{equation*}
  \begin{split}
    R(x)=R_0(x)+R_1(x)h_1+R_2(x)h_2+R_3(x)h_3\geq\\
       (R_0(x)+R_1(x)+R_2(x)+R_3(x))h_1\geq 0
  \end{split}
\end{equation*}
since, in $[\frac 14,\frac 12]$, $R_0(x)+R_1(x)+R_2(x)+R_3(x)\geq 0$.

Finally, in the interval $[\frac 12,1]$, by choosing $\alpha(x)=-3+6x$, we have $R_0(x)\geq 0,  R_1(x)> 0, R_2(x)\leq0, R_3(x)\geq 0$, $R_2(x)+R_3(x)\geq0$.
\begin{equation*}
  \begin{split}
    R(x)=R_0(x)+R_1(x)h_1+R_2(x)h_2+R_3(x)h_3\geq\\
        R_0(x)+R_1(x)h_1+(R_2(x)+R_3(x))h_2\geq 0.
  \end{split}
\end{equation*}

\end{proof}

Lemma \ref{minimal} and Lemma \ref{unit} imply that $Q(z)$ has a real root $<-1$ and the others in the unit disc. Therefore  by mimicking the case of two parameters, we have

\begin{thm}
For every triple of integers $(h_1,h_2,h_3)$, with $0<h_1\le h_2\le h_3$,  the polynomials $P^{(n)}_{(h_1,h_2,h_3)}(x)$ of the family with   Chebyshev coordinates $$[1,-h_1,h_2,-h_3,h_1,-h_2,h_3,\ldots,-h_1,h_2,-h_3,h_1,-h_2,h_3,1]$$ are hyperbolic, with all roots in the interval $[-2,2]$, except for one, which we denote  $x_m$ and which is less than $-2$; moreover $\displaystyle\lim_{n\to \infty} x_m=z_0+z_0^{-1}$, where $z_0$ is the smallest real root of the polynomial  $z^3+h_3z^{2}-h_2z+h_1+1$, while the largest root $x_M$ approaches $2$.
\end{thm}

\begin{rem}
  It can be shown that the limit of the span, which is $2-z_0-z_0^{-1}$, is the largest root of
  the resultant with respect to $z$ of $z^3+h_3z^{2}-h_2z+h_1+1$ and $(x-(2-z-z^{-1}))z$:
  \begin{equation*}
    \begin{split}
      -\left(h_1+1\right) x^3+\left[6 \left(h_1+1\right)-h_2+h_1 h_3+h_3\right]
   x^2\\
    -\left[9 \left(h_1+1\right)-h_1 h_2-h_2-h_2 h_3+h_3-4 \left(h_2-h_1
   h_3-h_3\right)\right] x\\
   +\left(h_1-h_2+h_3+2\right){}^2
    \end{split}
  \end{equation*}

\end{rem}

\begin{rem}
  Elementary calculations also show that $\displaystyle\lim_{n\to \infty} x_m$ is the negative root of the polynomial
$$\left(h_1+1\right) x^3+\left(\left(h_1+1\right) h_3-h_2\right)
   x^2$$
$$-\left(\left(h_1+1\right) \left(h_2+3\right)+\left(h_2-1\right) h_3\right)
   x$$
$$+\left(h_2+1\right){}^2+\left(h_1-h_3+1\right){}^2$$
\end{rem}

\end{document}